\documentclass[12pt]{article}
\usepackage{amsopn}
\usepackage{amssymb}
\usepackage{amsthm}
\usepackage{amsmath}
\usepackage{booktabs}
\usepackage[british]{babel}
\usepackage{mathptmx}

\theoremstyle{plain}
\newtheorem{thm}{Theorem}
\newtheorem{prop}{Proposition}

\theoremstyle{definition}
\newtheorem{defn}{Definition}

\theoremstyle{remark}
\newtheorem*{rem}{Remark}

\DeclareMathOperator{\im}{im}
\DeclareMathOperator{\tr}{tr}
\DeclareMathOperator{\SL}{SL}
\DeclareMathOperator{\GL}{GL}
\DeclareMathOperator{\PSL}{PSL}
\DeclareMathOperator{\PGL}{PGL}
\DeclareMathOperator{\Aut}{Aut}
\DeclareMathOperator{\Gal}{Gal}
\DeclareMathOperator{\Disc}{Disc}
\DeclareMathOperator{\Frob}{Frob}

\newcommand{\CC}{\mathbb{C}}
\newcommand{\FF}{\mathbb{F}}
\newcommand{\PP}{\mathbb{P}}
\newcommand{\QQ}{\mathbb{Q}}
\newcommand{\ZZ}{\mathbb{Z}}
\newcommand{\Qbar}{\overline{\QQ}}
\newcommand{\Flbar}{\overline{\FF}_\ell}
\newcommand{\GQ}{G_{\QQ}}
\newcommand{\GQp}{G_{\QQ_p}}

\newcommand{\Qlunr}{\QQ_\ell^{\rm unr}}

\newcommand{\ov}[1]{\overline{#1}}

\newcommand\mat[4]{
  \left(
  {#1 \atop #3}
  \thinspace
  {#2 \atop #4}
  \right)
}

\newcommand\vect[2]{
  \left({#1 \atop #2} \right)
}

\makeatletter
\def\eqalign#1{\null\,\vcenter{\openup\jot\m@th
  \ialign{\strut\hfil$\displaystyle{##}$&$\displaystyle{{}##}$\hfil
      \crcr#1\crcr}}\,}                                            
\makeatother

\newlength{\binnenmarge}
\newlength{\buitenmarge}
\newlength{\bovenmarge} 
\newlength{\ondermarge} 

\title{Modular forms applied to\\ the computational inverse Galois problem}
\author{Johan Bosman\thanks{Partially supported by DFG grant SFB/TR 45 
and Marie Curie FP7 grant 252058}}
\date\today

\begin{document}
\maketitle
\begin{abstract}\noindent
For each of the groups $\PSL_2(\FF_{25})$, $\PSL_2(\FF_{32})$, 
$\PSL_2(\FF_{49})$, $\PGL_2(\FF_{25})$, and $\PGL_2(\FF_{27})$,
we display the first explicitly known polynomials over $\QQ$ having
that group as Galois group.  
Each polynomial is related to a Galois representation associated
to a modular form.
We indicate how computations with modular Galois representations 
were used to obtain these polynomials.  For each polynomial,
we also indicate how to use Serre's conjectures to determine the modular
form giving rise to the related Galois representation.
\end{abstract}
\section{Introduction}
The inverse Galois problem dates from the 19th century and asks whether every 
finite group is isomorphic
to the Galois group of a finite field extension of $\QQ$.  Though solving this
problem in general seems currently out of reach, it can be an 
interesting and tractable task to solve it for certain specific types of groups.
A further challenge
is to actually \emph{compute} 
polynomials in $\QQ[x]$ having a prescribed permutation group as Galois group,
rather than merely pointing out a formal proof of their existence.
We speak about permutation groups here because the Galois group of a
separable polynomial acts on the roots of the polynomial inside 
a splitting field.

A lot of pioneering work on the computational inverse Galois problem has been
done by J\"urgen Kl\"uners and Gunter Malle.  Their joint publication
\cite{KlMa} gives a solution to this problem for all transitive
permutation groups of degree up to $15$.  
Gunter Malle has also published a method to compute polynomials for the groups 
$\PSL_2(\FF_p)$ and $\PGL_2(\FF_p)$ where $p$ is a prime number satisfying 
$(\frac{n}{p})=-1$ for at least one $n\in\{2,3,5,7\}$ (see~\cite{Ma1}).  
In the present paper, we will display
polynomials for several groups of the type $\PSL_2(\FF_q)$ and $\PGL_2(\FF_q)$ with 
$q$ a perfect prime power.  This can be seen as an extension of a previous 
result by the author for $q=16$ (see~\cite{BoSL2F16}).
Theoretical results stating that $\PSL_2(\FF_q)$ appears as
Galois group over $\QQ$ for `many' $q$ 
can be found in \cite{WieGal} and \cite{DieuWie}.

The groups in question are $\PSL_2(\FF_{25})$, $\PSL_2(\FF_{32})$, 
$\PSL_2(\FF_{49})$, $\PGL_2(\FF_{25})$, and $\PGL_2(\FF_{27})$.  
Polynomials 
with these groups as Galois groups are displayed at the end of the paper, 
in 
Section~\ref{section:polynomials}.  
In Section \ref{section:comput} we will indicate how modular forms can be 
used to compute these polynomials.  However, the computations do not give an
output that is guaranteed to be correct.
Techniques for verifying the 
Galois groups are discussed in Section~\ref{section:Galois}.  
In Section \ref{section:modverify} we will discuss how to apply 
Serre's conjectures to find the modular forms whose Galois representations
are attached to the polynomials.

\subsection{Notations and conventions}
For each field $k$ that is either a prime field or a completion of $\QQ$, 
we fix an algebraic closure $\ov{k}$ and we denote the absolute Galois group
by $G_k$.  Where this is useful or appropriate, 
we view algebraic extensions $K$ of $k$
as subfields of $\ov{k}$ and we identify $\ov{K}$ with~$\ov{k}$.  
We also fix for each (possibly infinite) prime $p$ 
an embedding $\Qbar\hookrightarrow\Qbar_p$ and 
use this embedding to view $\GQp$ as a subgroup of~$\GQ$.  The inertia
subgroup of $\GQp$ is denoted by~$I_p$.  

Representations of a group over a
field are assumed to be \emph{continuous}.  Two representations are considered 
isomorphic if they are isomorphic over the algebraic closures of 
their fields of definition.  

By abuse of notation, we identify an abelian character of $\GQ$ with the 
Dirichlet character that is attached to it by the Kronecker-Weber theorem.

\section{Some remarks on the computations}\label{section:comput}
Let $N$ be a positive integer, let 
$f=\sum_{n\geq 1} a_n(f)q^n \in S_2(\Gamma_1(N))$ be a newform, and denote its 
nebentypus character by $\varepsilon_f$.  Let $\ell$ be a prime number, 
let $\lambda\mid\ell$ be a prime of the coefficient field of $f$, and let 
$\FF_\lambda$ denote the residue field of $\lambda$.  Then there exists a 
representation $\rho = \ov\rho_{f,\lambda}\colon\GQ\to\GL_2(\FF_\lambda)$ 
such that for all primes $p\nmid N\ell$ we have 
\begin{equation}\label{eq:charpolrho}
\operatorname{charpol}(\rho(\Frob_p))
\equiv x^2-a_p(f)x+\varepsilon_f(p)p\bmod \lambda
\end{equation}
for any Frobenius element in $\GQ$ attached to $p$.
For the purposes of this paper, 
we will restrict our attention to cases where $\rho$ is irreducible.
Composing $\rho$ with the canonical projection map 
$\GL_2(\FF_\lambda)\to\PGL_2(\FF_\lambda)$,
one obtains a projective representation 
$\tilde\rho\colon\GQ\to\PGL_2(\FF_\lambda)$.  
The fixed field inside $\Qbar$ of $\ker(\tilde\rho)$ has Galois group 
isomorphic to $\im(\tilde\rho)$.  Via Galois theory, the permutation 
action of $\im(\tilde\rho)$ on $\PP^1(\FF_\lambda)$ can be described by giving
a suitable polynomial ${P_{f,\lambda}\in\QQ[x]}$ of degree $\#\FF_\lambda+1$ 
that has $\Qbar^{\ker(\tilde\rho)}$ as splitting field.
Such a $P_{f,\lambda}$ has Galois group isomorphic to $\im(\tilde\rho)$, 
thus computing $P_{f,\lambda}$ would explicitly realize $\im(\tilde\rho)$
as Galois group over $\QQ$.

What would be useful here, is a test that ensures 
$\im(\tilde\rho)\supset\PSL_2(\FF_\lambda)$ beforehand.
Define, for any field $k$, a function $\theta\colon\PGL_2(k)\to k$ as follows:
\begin{equation}\label{eq:theta}
\theta=\theta_k\colon\PGL_2(k)\to k,\quad
\ov{\gamma}\mapsto\frac{\tr(\gamma)^2}{\det\gamma}
\end{equation}
Then we have the following proposition, 
which can be verified straightforwardly using Dickson's classification of the 
possible subgroups of $\PSL_2(\FF_q)$ \cite[Section~III.6]{Suz}.
\begin{prop}
Let $q\geq 4$ be a prime power and let 
$\theta\colon\PGL_2(\FF_q)\to\FF_q$ be as in \eqref{eq:theta}.  
Let $G$ be a subgroup of $\PSL_2(\FF_q)$.  Then we have $G=\PSL_2(\FF_q)$ 
if and only if $\theta(G)$ is the full set $\FF_q$.\qed
\end{prop}
Coefficients and characters of modular forms can be computed (see \cite{StComp}).
For a given small prime $p\nmid N\ell$, we can thus use \eqref{eq:charpolrho} to 
compute $\tr(\rho(\Frob_p^n))$ and $\det(\rho(\Frob_p^n))$ for any $n\geq 0$.
If $\det(\rho(\Frob_p^n))$ happens to be a square in $\FF_\lambda$, 
then we compute $\theta(\tilde\rho(\Frob_p^n))\in\FF_\lambda$, 
where $\theta$ is as in the above proposition. If, after having tried several 
small primes $p$, we have met every element of $\FF_\lambda$, then we know that 
the projective image of $\rho$ contains $\PSL_2(\FF_\lambda)$. 
The character $\varepsilon_f$ can now be used to decide between 
$\im(\tilde\rho)=\PSL_2(\FF_\lambda)$ and 
$\im(\tilde\rho)=\PGL_2(\FF_\lambda)$, as we have 
$\det\rho = \varepsilon_f\chi_\ell$, where $\chi_\ell$ denotes the mod $\ell$
cyclotomic character of $\GQ$.

Once we have found a modular form $f$ giving rise to a desired Galois group, we
let $X$ be a modular curve over which the modular form $f$ lives.  
Typically, $X$ can be taken to be $X_1(N)$, where $N$ is the level of $f$,
 but sometimes a quotient of 
$X_1(N)$ of smaller genus may work as well.  The representation $\rho$ can be 
described as the action of $\GQ$ on a certain Hecke-invariant subspace 
$V_{f,\lambda}$ of $\operatorname{Jac}(X)(\Qbar)[\ell]$
(see for instance \cite[Sections 3.2 and~3.3]{RiSt}).

Computing modular Galois representations is the subject of both 
\cite{TheBigBook} and \cite{Peter}. 
In \cite{BoComp}, which is part of \cite{TheBigBook},
the author describes in detail how a polynomial $P_{f,\lambda}$
can be computed in practice.  
We will now give a very brief overview of the methods described there.
The general strategy is to first 
compute the points
in $V_{f,\lambda}$ over $\CC$ to a high precision (sometimes a few thousands 
of decimals are required) and to use this to compute a real approximation 
$P'_{f,\lambda}$ of $P_{f,\lambda}$.  
The next step is to approximate $P'_{f,\lambda}$
by a polynomial $P''_{f,\lambda}\in\QQ[x]$ of relatively small height and 
to apply some heuristics to decide whether the used precision was high enough
to make $P''_{f,\lambda}=P_{f,\lambda}$ likely to be true.  If the heuristic
test are passed, then we compute the maximal order ${\cal O}_K$ 
of the number field $K$ defined by $P''_{f,\lambda}$ and search for elements 
$\alpha\in{\cal O}_K$ of small height with $\QQ(\alpha)=K$, using the 
lattice structure that ${\cal O}_K$ gets using all complex embeddings
$K\hookrightarrow\CC$.
The minimal polynomial $P'''_{f, \lambda}$ of $\alpha$ 
will then have small coefficients.  Let us emphasize that,
because of the numerical nature of the calculations, correctness of the 
polynomials is not guaranteed at this stage of the computation,
so verification methods as in 
Sections \ref{section:Galois} and \ref{section:modverify} are essential.

Several software packages were used to carry out the 
computations described there.  The author used a {\sc Sage} \cite{Sage} 
implementation for the numerical approximation and the computation of initial 
candidate polynomials $P''_{f,\lambda}$.  After this, a combination of computations 
in orders of number fields using {\sc Magma} \cite{Magma}
and the function {\tt polredabs} of {\sc Pari} \cite{Pari} was used to obtain 
polynomials $P'''_{f,\lambda}$ with coefficients of tractable size.

\section{Verifying the Galois groups}\label{section:Galois}
Nowadays, {\sc Magma} %\cite{Magma} 
can compute the Galois group of a square-free polynomial 
$P\in \QQ[x]$ of arbitrary degree, thanks to an implementation of 
Claus Fieker and J\"urgen Kl\"uners.  
It represents its output as a permutation group acting on the roots
of $P$ in a suitably chosen $p$-adic field.  These roots are in turn
represented by sufficiently accurate $p$-adic approximations.  
The author fed the polynomials displayed in Section \ref{section:polynomials} to 
this implementation and indeed 
all their Galois groups could be rigorously computed without difficulties.

Let $P\in\QQ[x]$ be one of the polynomials from 
Section~\ref{section:polynomials}.
We have computed $\Gal(P)$ as permutation group of degree $\deg(P)$.
Let $G$ be the group that in Section \ref{section:polynomials} is claimed 
to be isomorphic to $\Gal(P)$.  
In all cases, $G$ is $\PSL(\FF)$ or $\PGL(\FF)$ acting on $\PP^1(\FF)$ for 
some finite field $\FF$.  We want to verify $\Gal(P)\cong G$.
Computing an explicit isomorphism between $\Gal(P)$ and $G$ can be an
extensive task in some of the cases and we can indeed do better.  

For any field $k$, the group $\PSL_2(k)$ is doubly transitive.
It is easy to verify that $\Gal(P)$
too is doubly transitive and that it has order equal to $\#G$.
The doubly transitive permutation groups
have been classified, see for instance \cite[Theorem~5.3]{Cam}.  The cited 
reference sums up results using the classification of finite simple groups, 
though prior to this, Charles Sims
had already performed an unpublished computation of 
the primitive permutation groups up to degree $50$ that does not rely
on this classification.
Consulting the classification of doubly transitive groups,
one can verify that 
$\Gal(P)$, given its degree $q+1$ and its order,  
contains $\PSL_2(\FF_q)$ as simple normal subgroup.  
We can already conclude from this information that each of the 
groups $\PSL_2(\FF_{25})$, $\PSL_2(\FF_{32})$, $\PSL_2(\FF_{49})$ 
and $\PGL_2(\FF_{27})$
is the unique doubly transitive permutation group of its degree and order.
In addition, of the $3$ doubly transitive permutation groups of degree $26$ 
and order $15600$, $\PGL_2(\FF_{25})$ is the only one with $27$ conjugacy classes.
These observations cover all cases occurring in Section \ref{section:polynomials}.
So we see that we can identify $G$ easily by computing some simple invariants 
that it has as a permutation group, 
without having to find explicit isomorphisms $G\cong\Gal(P)$.

A reader who has no access to a computer running {\sc Magma} may still wish to 
verify the Galois groups.  Let us describe briefly how this can be done. 
Let $P\in\QQ[x]$ be one of the polynomials given in Section 
\ref{section:polynomials} and let $\Omega(P)\subset\Qbar$ be the 
set of roots of $P$.  The double transitivity of $\Gal(P)$ can be 
verified as follows.  Compute the resolvent
$Q(x):=\prod_{(\alpha, \beta)\in\Omega(P)^2}(x-a\alpha-b\beta)\in\QQ[x]$ for 
small integers $a$ and $b$ such that the roots of $Q$ are distinct.
Then $\Gal(P)$ is doubly transitive if and only if $Q$ is irreducible.
The polynomial $Q$ can be computed symbolically, see \cite[Chapter~3]{Soi}.
Using the classification given in \cite[Theorem~5.3]{Cam} one can determine the 
lattice of doubly transitive permutation groups of given degree.  With this 
lattice at hand, one could apply \cite[Algorithm~6.1]{GeKl} to determine 
the Galois group.

\section{Further verifications}\label{section:modverify}
Serre's conjectures \cite[(3.2.3$_?$) and (3.2.4$_?$)]{SeConj} state that 
every absolutely irreducible odd two-dimensional representation 
of $\GQ$ over a finite field is associated to a cuspidal newform in a sense 
similar to~\eqref{eq:charpolrho}.  
A Galois representation is called \emph{odd} if the 
determinant of a complex conjugation equals $-1$. 
Serre's conjectures have been proved by subsequent work of 
many people, among whom Khare, Wintenberger, and Kisin did the final steps 
(see \cite[Theorems 1.2 and~9.1]{KhaWi} and \cite[Theorem~0.1]{Kisin}).
In the present section, we will point out 
how Serre's conjectures can be used to determine for each polynomial 
in Section \ref{section:polynomials} a modular form to which it is attached.

\subsection{Number field attached to projective representation}\label{subsec:nf}
Let $q$ be a prime power and let the permutation group $G$ be either 
$\PSL_2(\FF_q)$ or $\PGL_2(\FF_q)$ acting on 
$\PP^1(\FF_q)$ in the standard way.

A projective representation $\rho\colon\GQ\twoheadrightarrow G\subset\PGL_2(\FF_q)$ 
makes $\GQ$ act transitively on $\PP^1(\FF_q)$. Via Galois theory this 
action corresponds to a number field $K$ of degree $q+1$ over $\QQ$ such 
that the normal closure of $K/\QQ$ in $\Qbar$ has Galois group $G$.

On the other hand, from Dickson's classification of the subgroups of 
$\PSL_2(\FF_q)$ \cite[Section~III.6]{Suz} one can derive that any subgroup of $G$ 
of index $q+1$ is conjugate to the subgroup represented by the upper triangular 
matrices and therefore that any transitive permutation action of $G$ on a set of 
size $q+1$ is isomorphic to the standard action on $\PP^1(\FF_q)$.  

Let now $K$ be a number field of degree $q+1$ over $\QQ$ such that its 
normal closure $L$ over $\QQ$ has Galois group isomorphic to $G$.  
Choosing an isomorphism
$\Gal(L/\QQ)\cong G$ defines a
projective representation $\tilde\rho_K\colon\GQ\to \PGL_2(\FF_q)$
with image $G$.  The number
field $K$ corresponds to the stabilizer of a point of $\PP^1(\FF_q)$.
The automorphism group of $G$ is $\PGL_2(\FF_q)\rtimes\Aut(\FF_q)$, where 
$\PGL_2(\FF_q)$ acts by conjugation and $\Aut(\FF_q)$ acts on matrix entries.
The representation $\tilde\rho_K$ is unique up to an automorphism of $G$
and is thus well defined up to a choice of coordinates of $\PP^1(\FF_q)$ and
an automorphism of $\FF_q$.

\subsection{Level and weight}
If $\ell$ is a prime number, then an irreducible representation 
$\rho\colon\GQ\to\GL_2(\Flbar)$ has a \emph{Serre level} and a \emph{Serre weight}.
The Serre level $N(\rho)$ is equal to the prime-to-$\ell$ part of the Artin 
conductor of $\rho$ (see \cite[Subsection~1.2]{SeConj}).
If $\rho$ is at most tamely ramified at a prime $p\not=\ell$, then 
$N(\rho) = 2 - \dim V^{I_p}$, where $V$ denotes the $2$-dimensional 
$\Flbar$-vector space with $\GQ$-action via $\rho$.  The Serre
weight $k(\rho)$ is described in terms of $\rho|_{I_\ell}$; for the definition, 
which depends on distinguishing several cases,
we refer to \cite[Section~4]{EdixWt}.
Serre's \emph{strong} conjecture states that $\rho$, if it is odd, 
is associated to a newform of level $N(\rho)$ and weight at most $k(\rho)$.  
There's a small caveat to be made here: 
in case $\ell=2$ and $k(\rho)=1$,
we may sometimes have to redefine $k(\rho)$ as $2$ in order
to make Serre's strong conjecture fully proven.

Let $\ell$ be a prime and let $\tilde\rho\colon\GQ\to\PGL_2(\Flbar)$ 
be an irreducible projective representation.  The various liftings of 
$\tilde\rho$ will have various Serre levels and Serre weights.  In view of 
the following theorem by Tate there exists a lifting $\rho$ of $\tilde\rho$ 
such that $k(\rho)$ and $v_p(N(\rho))$ for all $p$ are minimized simultaneously.  
\begin{thm}[{Tate, see \cite[Section~6]{SeWt1}}]
Let $k$ be an algebraically closed field and let 
$\tilde\rho\colon\GQ\to\PGL_2(k)$ be a projective representation.  
Then there exists a lifting $\rho\colon\GQ\to\GL_2(k)$ of $\tilde\rho$.  
If, furthermore, for each prime $p$ a lifting $\rho_p\colon\GQp\to\GL_2(k)$ of
$\tilde\rho|_{\GQp}$ is given with all but finitely many $\rho_p$ unramified, 
then a lifting $\rho$ of $\tilde\rho$ exists with $\rho|_{I_p}\cong\rho_p|_{I_p}$ 
for each $p$.
\end{thm}
\begin{defn}
Let $\ell$ be a prime and let $\tilde\rho\colon\GQ\to\PGL_2(\Flbar)$
be an irreducible projective representation.  
The \emph{Serre level} $N(\tilde\rho)$ and \emph{Serre weight} $k(\tilde\rho)$ of 
$\tilde\rho$ are defined to be minimum of the Serre levels, resp.\ weights, 
of all the 
liftings $\rho\colon\GQ\to\GL_2(\Flbar)$ of $\tilde\rho$. 
\end{defn}
From the definition of Serre weight given in \cite[Section~4]{EdixWt} one 
can immediately see that in general we have
$1\leq k(\tilde\rho)\leq \ell+1$.

The following proposition shows how $N(\tilde\rho)$
can be related to the ramification behaviour of the number field $K$ attached to it
in the case that $\tilde\rho$ is tamely ramified outside $\ell$.
\begin{prop}\label{prop:level}
Let $\ell$ be a prime number and let $\FF\subset \Flbar$ be a 
finite field.  Let $\tilde\rho\colon\GQ\to\PGL_2(\FF)$ be a
projective representation with $\im(\tilde\rho)\supset\PSL_2(\FF)$ and let $K$ 
be its attached number field, as described in Subsection \ref{subsec:nf}.
Let $p\not=\ell$ be a prime above which $K/\QQ$ is at most tamely ramified.  
Then the valuation $v_p(N(\tilde\rho))$ is at most $2$ and 
can be expressed as follows:
\[
v_p(N(\tilde\rho)) = 
\left\{
\begin{array}{ll}
0 & \mbox{if $K$ is unramified at $p$},\cr
1 & \mbox{if $K$ is ramified at $p$ but also has an unramified prime above $p$},\cr
2 & \mbox{if $K$ has no unramified primes above $p$}.
\end{array}
\right.
\]
\end{prop}

\begin{proof}
The local representation $\tilde\rho_p:=\tilde\rho|_{\GQp}$ 
factors through the tame quotient $G$ of $\GQp$, which is topologically 
generated by two elements $\sigma$ and $\tau$, 
where $\tau$ generates the inertia group $I_p$ and $\sigma$ corresponds to 
a Frobenius.  They satisfy the relation
\begin{equation}\label{tameram}
\sigma\tau\sigma^{-1} = \tau^p.
\end{equation}
Giving a tamely ramified lifting $\rho\colon\GQp\to\GL_2(\Flbar)$ of 
$\tilde\rho$ is equivalent to giving liftings of 
$\tilde\rho(\sigma)$ and $\tilde\rho(\tau)$ 
that are compatible with \eqref{tameram}.
The element $\tilde\rho(\tau)$ of $\PGL_2(\FF)$ can have four possible shapes: 
it can be trivial, 
the reduction of a non-trivial unipotent element of $\GL_2(\FF)$, 
the reduction of a non-scalar split semi-simple element of $\GL_2(\FF)$,
or the reduction of a non-split semi-simple element of $\GL_2(\FF)$.

If $\tilde\rho(\tau)$ is trivial then $K$ is unramified at $p$.  
We can choose an unramified lift
$\rho$ of $\tilde\rho$ by taking $\rho(\tau)$ to be trivial and $\rho(\sigma)$ 
any lifting of $\tilde\rho(\sigma)$.  In that case we have $v_p(N(\rho))=0$.

If $\tilde\rho(\tau)$ is represented by a 
non-trivial unipotent matrix, then after conjugation we may assume 
that $\tilde\rho(\tau)$ is 
represented by $\mat{1}{1}{0}{1}$.  
From \eqref{tameram} one can derive that $\tilde\rho(\sigma)$ is represented
by a matrix of the form $\mat{p}{b}{0}{1}$.  
Putting $\rho(\sigma)=\mat{p}{b}{0}{1}$ and $\rho(\tau)=\mat{1}{1}{0}{1}$ gives
a lifting of conductor-exponent $1$ at $p$, which is minimal as $\tilde\rho$ 
is ramified.
The action of $\tilde\rho_p$ on $\PP^1(\FF)$ has exactly 1 orbit of length~1. 
This corresponds to 1 unramified prime of degree 1 above $p$ in $K$.  
 
In the split semi-simple case, 
after conjugation, we assume that $\tilde\rho(\tau)$ is represented by 
$\mat{\alpha}{0}{0}{1}$.
Using \eqref{tameram}, one can readily verify that $\tilde\rho(\sigma)$ 
is represented 
by a diagonal or anti-diagonal matrix.
Taking any lifting of $\tilde\rho(\sigma)$ and
$\rho(\tau)= \mat{\alpha}{0}{0}{1}$ yields $v_p(N(\rho))=1$.
The action of $\tilde\rho_p(\tau)$ on $\PP^1(\FF)$ does have fixed points,
which means that $K$ has unramified primes above $p$.  

We are left with the non-split semi-simple case.  
As a wildly ramified lifting of
$\tilde\rho_p$ has conductor-exponent at least $2$, we concentrate on searching 
for tamely ramified liftings.  
After a conjugation over $\Flbar$ we may suppose that $\tilde\rho(\tau)$
is represented by a matrix
$M=\mat{\lambda}{0}{0}{\lambda^p}$ with $\lambda$ quadratic over $\FF$.  
We assume $\tr M\not = 0$, as $\tr M=0$ is
included in the split semi-simple case.
From \eqref{tameram} it now follows that
$\tilde\rho(\sigma)$ is represented by an anti-diagonal matrix.  
Hence, any lifting $\rho(\tau)$ has to be a diagonal matrix that is conjugate
but not equal to its $p$-th power.  This is only possible if $\rho(\tau)$ is 
in $\FF^\times M$.  Any choice of $\rho(\tau)$ in this set and any lifting
of $\tilde\rho(\sigma)$ defines a tamely ramified lifting of $\tilde\rho_p$ 
of conductor-exponent $2$.
In this case, the inertia group has no fixed points in $\PP^1(\FF)$ so $K$ has 
no unramified primes above $p$.
\end{proof}
It should certainly be possible to generalize the above proposition
to representations that are wildly ramified at some place outside $\ell$ 
using, for instance, the results from \cite[Subsections 5.1 and~5.2]{KimVer}.
However, this may become somewhat elaborate and does not apply to 
the polynomials in Section \ref{section:polynomials}, so we will not do this here.

For the Serre weight we have the following proposition:
\begin{prop}\label{prop:weight}
Let $q$ be a power of a prime $\ell$ and let 
$\tilde\rho\colon\GQ\to\PGL_2(\FF_q)$ be a projective representation with 
$\im(\tilde\rho)\supset\PSL_2(\FF_q)$.  Let $K$ be the number field attached  
to $\tilde\rho$, as described in Subsection \ref{subsec:nf}.  
If $m$ is a positive integer such that $K$ has a prime of wild ramification 
index $\ell^m$ above $\ell$, then we have
\[
k(\tilde\rho) = 
1 + \left\lceil
\frac{(\ell-1)\ell^m}{(\ell^m-1)q}
\cdot\big(
v_\ell(\Disc(K/\QQ)) - q + 1
\big)\right\rceil.
\]
\end{prop}
\begin{proof}
For $q=\ell$, the author gave a similar formula in 
\cite[Corollary 7.2.8]{BoLvl1}; we will generalize the derivation given there.

Put $k=k(\tilde\rho)$ and $d=\gcd(k-1, \ell-1)$ and consider a lifting
$\rho$ of $\tilde\rho$ of weight $k$. 
From the definition of weight it follows that $\rho|_{I_\ell}$ is isomorphic
to a non-split representation of the form
$\mat{\chi_\ell^{k-1}}{*}{0}{1}$.  We see that
$\ker(\tilde\rho|_{I_\ell}) = \ker(\rho|_{I_\ell})$.
Let $L/\Qlunr$ be the fixed field of $\ker(\rho|_{I_\ell})$.
Denote by $P<\FF_q$ be the subgroup that is the image of the upper right entry.
Then the action of $I_\ell$ on $\PP^1(\FF_q)$ has orbits $\{\vect{1}{0}\}$,
$\{\vect{x}{0}\colon x \in P\}$, and apart from these two, only orbits
of full length~$\frac{\ell-1}{d}\#P$.  It follows that $\#P=\ell^m$ and 
that the \'etale algebra over $\Qlunr$ attached to the Galois action
of $I_\ell$ on $\PP^1(\FF_q)$ is isomorphic to
\[
\Qlunr \times L' \times L^{d(q-\ell^m)/((\ell-1)\ell^m)},
\]
where $L'$ is a subfield of $L$ of index $(\ell-1)/d$.
If ${\cal D}$ denotes the different, then 
\[
v_\ell({\cal D}(L/L')) = \frac{\frac{\ell-1}{d} - 1}{\frac{\ell-1}{d}\ell^m}
= \frac{\ell-1-d}{(\ell-1)\ell^m}.
\]
From this we obtain
\begin{eqnarray}\label{discKQ}
v_\ell(\Disc(K/\QQ)) &=&
\ell^m(v_\ell({\cal D}(L/\Qlunr)) - v_\ell({\cal D}(L/L')))
 + (q - \ell^m)v_\ell({\cal D}(L/\Qlunr))\cr
 &=& q\cdot v_\ell({\cal D}(L/\Qlunr)) - \frac{\ell-1-d}{\ell-1}
\end{eqnarray}

Now we invoke \cite[Theorem 3]{MoTa}, which expresses the different of $L/K$ 
in terms of $k$.  If $k\leq \ell$, this theorem says 
$
v_\ell({\cal D}(L/\Qlunr))
 = 
1 + \frac{k-1}{\ell-1} - \frac{k-1-d}{(\ell-1)\ell^m}.
$
Plugging this into \eqref{discKQ} and rewriting yields:
\[
k - 1 \,=\,
\frac{(\ell-1)\ell^m}{(\ell^m-1)q}
\cdot\big(
v_\ell(\Disc(K/\QQ)) - q + 1
\big)
\,+\, \frac{d(q-\ell^m)}{(\ell^m-1)q}.
\]
As we always have $1 \leq d \leq \ell-1$ and $\ell\leq\ell^m\leq q$,
this implies the formula in the proposition.

In the case $k=\ell + 1$, \cite[Theorem 3]{MoTa} reads 
$v_\ell({\cal D}(L/\Qlunr)) = 2 + \frac{1}{(\ell-1)\ell} - \frac{2}{(\ell-1)\ell^m}$
so that we obtain
\[
k-1 \,=\, \ell \,=\, 
\frac{(\ell-1)\ell^m}{(\ell^m-1)q}
\cdot\big(
v_\ell(\Disc(K/\QQ)) - q + 1
\big)
\,+\, \frac{(\ell-1)(\ell^{m-1} - 1) + 1}{\ell^m - 1} - \frac{\ell^m}{q(\ell^m -1)},
\]
which implies the formula in the proposition as well.
\end{proof}
If $\tilde\rho$ is tamely ramified at $\ell$, 
it seems less clear how to give a direct way of 
computing the Serre weight of a projective representation in terms of the 
number field attached to it.  However, in that case 
we always have the upper bound 
$k(\tilde\rho)\leq (\ell + 3)/2$, which is in practice sufficient to 
reduce the number of possibilities to a tractable size.

\subsection{Finding the correct form}\label{subsec:form}
Once we know the level $N$ and weight $k$ of a given representation
$\tilde\rho\colon\GQ\to\PGL_2(\FF_q)$, 
with $q\geq 4$ a power of a prime $\ell$,
we wish to find an actual newform giving rise to it.  As everywhere in this
paper, we assume $\im(\tilde\rho)\supset\PSL_2(\FF_q)$.  
We also assume that $N$ is square-free and  
$k\not=1$; both conditions hold for all
representations attached to the polynomials in Section 
\ref{section:polynomials}.

To apply Serre's conjectures and conclude that there is a modular form
associated with $\tilde\rho$, we have to verify that $\tilde\rho$ is odd.
If $q$ is a power of $2$, this is automatic, so we assume $q$ is odd.
In that case, $\tilde\rho$ is odd if and only if the action of a complex
conjugation on $\PP^1(\FF_q)$ is non-trivial, which holds if and only if the number field
attached to $\tilde\rho$ is not totally real.  
None of the polynomials in Section \ref{section:polynomials} defines a totally
real number field hence all our representations are odd. 

We can list all the newforms in $S_k(\Gamma_1(N))$; 
the book \cite{StComp} deals extensively with the computational 
ingredients needed for this task.  If we can eliminate all but one of them,
then by Serre's conjectures, the remaining one must give rise to $\tilde\rho$.
We will therefore list a few tricks that can be used to prove 
$\tilde\rho\not\cong\tilde\rho_{f,\lambda}$.  We do not claim that these
tricks are in any way sufficient for the elimination process in general, 
but they do turn out to be sufficient for the 
representations occurring in this paper.

Let $f=\sum a_nq^n\in S_k(\Gamma_1(N))$ 
be a newform and consider its coefficient
field $K_f$.  The field $K_f$ is either a totally real field or 
a totally imaginary quadratic extension of a totally real field.  
Its subfield $F_f$ generated by all elements $a_p^2/\varepsilon_f(p)$ 
for $p\nmid N$ is totally real, as can be seen from properties of the
Petersson inner product.  In many cases, it turns out that
there is a small prime $p$ 
such that $\QQ(a_p^2/\varepsilon_f(p))$ is the maximal totally 
real subfield of $K_f$ and thus equal to $F_f$.
In view of \eqref{eq:theta}, the following proposition
is immediate.

\begin{prop}
With the notations and assumptions from above, 
let $\lambda$ be a prime of $K_f$, 
lying over the prime $\lambda'$ of $F_f$. 
Assume that there is a prime $p\nmid N\ell$ such that
${\cal O}_{F_f}/\lambda'$ is generated by $a_p^2/\varepsilon_f(p)$
over $\FF_\ell$.
If there is
an embedding $\FF_q\hookrightarrow{\cal O}_{K_f}/\lambda$
giving rise to an isomorphism $\tilde\rho\cong\tilde\rho_{f,\lambda}$,
then we have an isomorphism ${\cal O}_{F_f}/\lambda'\cong \FF_q$. \qed
\end{prop}
\begin{rem}
The set of primes $p$ with $F_f=\QQ(a_p^2/\varepsilon_f(p))$ has density $1$
(see \cite[Theorem~1]{KooStWie}), so one can expect that the condition on 
$\lambda'$ is almost always satisfied.
In any case we do have an embedding 
$\FF_q\hookrightarrow {\cal O}_{F_f}/\lambda'$.
\end{rem}
So for given $f$, we do the following.  Determine all primes $\lambda'$ 
of $F_f$ above $\ell$ whose degree is a multiple of $[\FF_q:\FF_\ell]$.
Check whether for each $\lambda'$ the residue field of $\lambda'$ 
is generated by
$a_p^2/\varepsilon(p)$ for some small prime $p$.  If this is the case,
then we can reject $f$ if none of the $\lambda'$ have degree $[\FF_q:\FF_\ell]$.
This criterion can often be verified without difficulty. 

Using the above trick we can already eliminate a lot of pairs $(f,\lambda)$,
but some of them remain.  For a small prime $p\nmid N\ell$ we can compute
\[
\theta(\tilde\rho_{f,\lambda}(\Frob_p)) \equiv 
\frac{a_p^2}{\varepsilon_f(p)p^{k-1}}
\bmod\lambda
\]
and check whether it is equal to~$4$. 
For a given $\gamma\in\GL_2(\FF_q)$ we have $\theta(\gamma)=4$ 
if and only if the eigenvalues of $\gamma$ are equal if and only if the action
of $\gamma$ on $\PP^1(\FF_q)$ is either trivial or has exactly $1$ fixed point.
Thus for $p\nmid N\ell$, we have $\theta(\tilde\rho(\Frob_p))=4$ if and only
if $p$ either splits completely or has exactly $1$ factor of degree $1$ over
the number field attached to $\tilde\rho$.  So we check for many small 
primes $p\nmid N\ell$ whether $\theta(\tilde\rho_{f,\lambda}(\Frob_p))$
and $\theta(\tilde\rho(\Frob_p))$ are simultaneously equal to~$4$. 

Of the polynomials in Section \ref{section:polynomials}, only one can stand up to
 the above tricks, namely the one with Galois group 
$\PSL_2(\FF_{32})$.  So assume for the rest of the subsection 
that $q$ is a power of~$2$. We then have 
$\PGL_2(\FF_q)\cong\PSL_2(\FF_q)\cong\SL_2(\FF_q)$.  So $\tilde\rho$ is a lifting
of itself and under the assumption that $N(\tilde\rho)$ be square-free, this
lifting is minimal.  
Thus we are dealing with a representation whose image lands inside $\SL_2(\FF_q)$.
A theorem of Buzzard \cite[Corollary~2.7]{Buzz}
now shows that there exists a newform $f$
of trivial nebentypus giving rise to $\tilde\rho$, so we may assume
$f\in S_k(\Gamma_0(N))$.

Applying this to our particular $\PSL_2(\FF_{32})$-polynomial, 
still two newforms $f$ and $f'$ in $S_k(\Gamma_0(N))$ are
not rejected by the above tricks.  So we have prime ideals $\lambda$
and $\lambda'$ of their respective coefficient fields
such that $f\bmod\lambda$ and $f'\bmod\lambda'$ could both
give rise to $\tilde\rho$.  In fact, it turns out that, after a suitable
isomorphism between the respective residue fields, $f\bmod\lambda$ is
equal to $f'\bmod\lambda'$.  This can be shown by verifying 
$\ov{a_n(f)}=\ov{a_n(f')}$ for $n$ up to the so-called Sturm bound 
$[\SL_2(\ZZ):\Gamma_0(N)]/6$, which is up to $n=27$ in our particular case
(see \cite[Theorem~1]{Stu}). 

\begin{rem}
The forms $f$ and $f'$ have eigenvalues $+1$ and $-1$ for the Atkin-Lehner
operator $W_N$.  In general, for $N$ prime and $k=2$, a result of Mazur 
\cite[Proposition~10.6]{MazEis} shows that the spectrum of the Hecke algebra 
attached to $S_k(\Gamma_0(N))$ is connected.  This implies the existence 
of certain congruences between newforms.  In particular, if the $+1$-eigenspace 
$S_2(\Gamma_0(N))^+$ of $W_N$ on $S_2(\Gamma_0(N))$ is non-zero, 
there always exists a newform in $S_2(\Gamma_0(N))^{+}$ that is congruent to
a newform in $S_2(\Gamma_0(N))^{-}$ modulo some prime $p$. 
We must have $p=2$ as $+1$ and $-1$ cannot be congruent modulo any other prime.
\end{rem}

\section{Polynomials}\label{section:polynomials}
This section displays the polynomials that were computed for this paper.
All but the last one of the subsections have a group as title.  
This group is claimed to be the Galois group of all the polynomials occurring 
in that subsection.  The final subsection contains information about the 
ramification properties of the polynomials.

\subsection{$\PSL_2(\FF_{25})$}
\small
\[\eqalign{
&x^{26} + 25x^{24} - 90x^{23} + 410x^{22} - 2174x^{21} + 7915x^{20} 
- 24445x^{19} + 82385x^{18} - 174360x^{17} + 340352x^{16}\cr&{} - 596725x^{15} 
+ 606925x^{14} - 845215x^{13} + 2199840x^{12}- 1523031x^{11} + 203295x^{10}
 - 2102590x^{9} \cr&{}+ 1804065x^{8} - 28770x^{7} - 35747x^{6} 
+ 674760x^{5} - 134800x^{4} + 150735x^{3} - 2885x^{2} + 64x - 5
}\]
\smallskip
\[
\eqalign{
&x^{26} - 11x^{25} + 45x^{24} - 240x^{23} + 1425x^{22} - 4005x^{21} + 12885x^{20}
 - 50435x^{19} + 53555x^{18} - 142870x^{17} \cr&{}+ 503050x^{16} + 1144115x^{15} 
 - 1778920x^{14} - 3596690x^{13} - 26810705x^{12} + 72895865x^{11} 
\cr&{} + 110135765x^{10} + 472613940x^{9} - 1155934625x^{8} - 4427715545x^{7} 
 - 7223127110x^{6} - 17420055270x^{5} \cr&{}+ 2907221810x^{4} - 16043305910x^{3} 
 + 21674938395x^{2} + 14749741397x - 14641021707
}\]
\smallskip
\[
\eqalign{
&x^{26} + 8x^{25} + 35x^{24} + 160x^{23} - 20x^{22} - 130x^{21} + 9095x^{20} 
- 13020x^{19} - 43680x^{18} + 302710x^{17}\cr&{} - 420530x^{16} - 654320x^{15} 
+ 4610695x^{14} - 8622900x^{13} + 2477755x^{12} + 22760620x^{11} \cr&{}- 74710515x^{10} 
+ 87489200x^{9} - 50319960x^{8} + 23366430x^{7} - 50415455x^{6} - 166077740x^{5} 
\cr&{}+ 289509200x^{4} + 186724650x^{3} - 452029570x^{2} + 159622636x - 103539627
}\]\normalsize

\subsection{$\PSL_2(\FF_{32})$}
\small
\[\eqalign{
&x^{33} - 16x^{32} + 108x^{31} - 396x^{30} + 980x^{29} - 3000x^{28} 
+ 12404x^{27} - 35920x^{26} + 52252x^{25} - 12200x^{24} 
\cr&{}- 56484x^{23} 
+ 54996x^{22} - 56164x^{21} + 101320x^{20} + 90972x^{19} - 226860x^{18} 
- 92456x^{17} - 106536x^{16}\cr&{} + 299784x^{15} + 681744x^{14} - 308904x^{13} 
- 863008x^{12} - 67040x^{11} + 431272x^{10} + 192632x^9 - 9696x^8 
\cr&{}- 1416x^7 - 2888x^6 - 6600x^5 + 2800x^4 + 696x^3 - 632x^2 + 68x - 16
}\]\normalsize

\subsection{$\PSL_2(\FF_{49})$}
\small
\[
\eqalign{
&x^{50} + 14x^{48} - 133x^{47} - 112x^{46} - 1295x^{45} - 378x^{44} - 5929x^{43}
- 20643x^{42} + 293209x^{41} + 906654x^{40} \cr&{}+ 6607265x^{39} 
+ 24040177x^{38} + 
 85681897x^{37}  + 473415579x^{36} + 1538779634x^{35} 
+ 5045381579x^{34} 
 \cr&{}+ 17364043354x^{33} + 49737600486x^{32} + 172099058782x^{31} 
+ 
 417122339060x^{30} + 1354316398652x^{29} \cr&{}+ 3528932603770x^{28} 
 + 7809511870860x^{27} + 28405678075796x^{26} + 42739389341075x^{25} 
\cr&{} +124662200818617x^{24} + 270315747916557x^{23} + 494771507303808x^{22} + 
1033886525397236x^{21} \cr&{}+ 2097338665080414x^{20} + 3256453013950549x^{19} + 
4773474576206007x^{18} \cr&{}+ 17415459260623270x^{17} - 10288584860072456x^{16} 
+ 69831032535759796x^{15} 
\cr&{}-73209901903545764x^{14} + 211847950318229554x^{13} 
- 484953362514826317x^{12} 
\cr&{}+ 1233922356068052511x^{11} - 2544352497290479589x^{10} 
+ 4386728967245371033x^9 
\cr&{}- 5890119836852604710x^8 + 5642420291895645202x^7 - 1522273119205843039x^6 
\cr&{}- 1608191474819379639x^5 - 985882169176584092x^4 
+ 2132948153097061258x^3 
\cr&{}- 165342791798420467x^2 - 71582764911979429x + 7908857674762849
}\]
\smallskip
\[\eqalign{
&x^{50} - 5x^{49} - 77x^{48} + 329x^{47} + 3843x^{46} - 13874x^{45} 
- 131754x^{44} + 430178x^{43} + 3242540x^{42} \cr&{} - 9995482x^{41} - 60873015x^{40} 
+ 191874340x^{39} + 812024087x^{38} - 2785593678x^{37} \cr&{}- 7329384580x^{36} 
+ 26805294425x^{35} + 51738967427x^{34} - 163941114631x^{33} 
- 414645913171x^{32} \cr&{}+ 759116077097x^{31} + 3774174093592x^{30} 
- 6482891887052x^{29} - 14580121639230x^{28} \cr&{}+ 12142740277948x^{27} 
+ 113966950745802x^{26} - 17806982973332x^{25} - 1405472958758232x^{24} 
\cr&{}+ 4595833892032558x^{23} - 9516541438774671x^{22} + 31704705422352872x^{21} 
\cr&{}- 116138484174279574x^{20} + 264431039635704172x^{19} - 494002668821182528x^{18} 
\cr&{}+ 1362351868639873993x^{17} - 3448748331607098429x^{16} 
+ 5694479877938233865x^{15} \cr&{}- 9894493925776418252x^{14} 
+ 22915549471984648416x^{13} - 39473448044982762734x^{12} 
\cr&{}+ 55367518860559248182x^{11} - 92744284275900788951x^{10} 
+ 144645718920022553002x^{9} \cr&{}- 202308373939366049761x^{8} 
+ 272248056577059876663x^{7} - 284654408160120598600x^{6} 
\cr&{}+ 310646322644102048632x^{5} - 467838768538599148516x^{4} 
+ 185233561060467551772x^{3} \cr&{}- 489713859491859418738x^{2} 
- 124345023465677984401x - 248368725729104252373
}\]\normalsize

\subsection{$\PGL_2(\FF_{25})$}
\small\[\eqalign{
&x^{26} - 5x^{25} + 15x^{24} - 30x^{23} + 65x^{22} - 510x^{21} + 1460x^{20} 
- 1520x^{19} - 5525x^{18} + 29065x^{17} - 48510x^{16} \cr&{}+ 56150x^{15} 
+ 74695x^{14} - 28595x^{13} + 124915x^{12} + 430280x^{11} + 555465x^{10} 
+ 318535x^{9} + 805335x^{8} \cr&{}+ 1621715x^{7} + 1764955x^{6} + 950255x^{5} 
+ 229675x^{4} - 10010x^{3} - 5560x^{2} + 1984x + 425}
\]\normalsize

\subsection{$\PGL_2(\FF_{27})$}
\small\[
\eqalign{
&x^{28} - 13x^{27} + 69x^{26} - 144x^{25} - 252x^{24} + 2451x^{23} 
- 6957x^{22} + 8433x^{21} + 8103x^{20} - 64617x^{19}
\cr&{} + 187452x^{18} - 406998x^{17} + 734271x^{16} - 1114407x^{15} 
+ 1436532x^{14} - 1653204x^{13} + 1777944x^{12} \cr&{}- 1653399x^{11} 
+ 1189149x^{10} - 767391x^{9} + 511770x^{8} - 130359x^{7} + 16974x^{6} 
- 106098x^{5} - 49980x^{4} \cr&{}- 20697x^{3} - 6915x^{2} - 989x - 529
}\]
\smallskip
\[\eqalign{
&x^{28} - 7x^{27} + 51x^{26} - 210x^{25} + 843x^{24} - 2343x^{23} + 6645x^{22} -
12666x^{21} + 26937x^{20} - 25680x^{19}\cr&{} + 42918x^{18} + 73236x^{17} - 28737x^{16}
+ 589764x^{15} + 137034x^{14} + 1898235x^{13} + 2535021x^{12}\cr&{} + 5783667x^{11}
+ 11729181x^{10} + 19459167x^{9} + 34925964x^{8} + 46972173x^{7} + 62946807x^{6}
\cr&{} + 54973245x^{5} + 47069826x^{4} + 10921458x^{3} + 884292x^{2} - 11509304x - 2199865
}\]
\smallskip
\[\eqalign{
&x^{28} - x^{27} - 18x^{25} - 45x^{24} - 231x^{23} + 588x^{22} + 1548x^{21} -
816x^{20} + 1785x^{19} - 810x^{18} + 10632x^{17} \cr&{}- 39222x^{16} + 24270x^{15} 
+ 127512x^{14} - 52701x^{13} + 34995x^{12} - 1002237x^{11} + 1837884x^{10} 
\cr&{}- 1396431x^{9} + 2893974x^{8} - 5570163x^{7} + 4814445x^{6} - 3440091x^{5} + 
4975905x^{4}\cr&{} - 4773414x^{3} + 1350804x^{2} + 46475x + 265837
}\]\normalsize

\subsection{Ramification of the fields}
The following table contains information on the ramification properties
of each of the number fields defined by a polynomial given above.  Each row
of the table corresponds to one of the polynomials, sorted on Galois group. 
In case we have more than one field with a given Galois group, the order 
given in the table corresponds to the order in which the polynomials 
are displayed above.  

For each of the fields, the absolute value of the discriminant 
and the decomposition types of the
ramified primes are given.  The notation for the decomposition type of a given
prime $p$ with respect to given number field $K$ is to be read
as follows.  If the primes of $K$ above $p$ are 
$\mathfrak{p}_1,\ldots,\mathfrak{p}_r$ and for each $i$ the inertia and ramification
degrees of $\mathfrak{p}_i$ are $f_i$ and $e_i$ respectively,
then we denote this by
the expression $f_1^{e_1}\cdots f_r^{e_r}$.  
In the cases where there are $n>2$ factors of the same type $f^e$, we abbreviate
this as $(f^e)^n$.

The columns labelled $N$ and $k$ indicate the level and weight of the associated
representation, according to Propositions \ref{prop:level} and
\ref{prop:weight} respectively.  The corresponding cusp form $f$, determined 
using the methods from Subsection~\ref{subsec:form}, is written down in the
following concise way.  In all cases it has turned out that the coefficient 
$a_2(f)$
generates the coefficient field and that giving $N$, $k$, and the minimal 
polynomial of $a_2$ does pin down the Galois orbit of $f$.  
So we specify $f$ by writing down the minimal polynomial of~$a_2(f)$.

%\newline
\begin{center}
\begin{tabular}{llr@{$\colon$}l@{ }r@{$\colon$}lrll}
\toprule
Group & Discr. & \multicolumn{4}{l}{Decomposition types} 
& $N$ & $k$ & Minimal pol. of $a_2$\cr
\midrule
$\PSL_2(\FF_{25})$ &
$5^{28}\cdot 29^{20}$ & 
$5$ & $1^1 1^5 1^{20}$ & $29$& $1^1 1^5 2^5 2^5$ & 29 & 2 
& $x^2+2x-1$\cr
& $5^{30}\cdot 41^{20}$ & $5$ & $ 1^1 1^{25}$ & $41$ & $ 1^1 5^5$ & 41 & 2
& $x^3+x^2-5x-1$\cr
& $5^{30}\cdot 43^{20}$ & $5$ & $ 1^1 1^{25}$ & $43$ & $ 1^1 1^5 4^5$ & 43 & 2
& $x^2-2$\cr
\midrule
$\PSL_2(\FF_{32})$ & 
$2^{62}\cdot 157^{16}$ & $2$ & $ 1^1 1^{32}$ & $157$ & $ 1^1 (2^2)^8$ & 157 & 2
%& $x^{5} + 5 x^{4} + 5 x^{3} - 6 x^{2} - 7 x + 1$\cr
& $x^5+\cdots+1,\,\, W_N(f)=f$\cr
\midrule
$\PSL_2(\FF_{49})$ 
& $7^{56}\cdot 23^{42}$ & $7$ & $ 1^1 1^{49}$ & $23$ & $ 1^1 1^7 3^7 3^7$ & 23 & 2
& $x^2 + x - 1$\cr
& $7^{56}\cdot 31^{42}$ & $7$ & $ 1^1 1^{49}$ & $31$ & $ 1^1 1^7 6^7$ & 31 & 2
& $x^{2} - x - 1$\cr
\midrule
$\PGL_2(\FF_{25})$ & 
$5^{30}\cdot 17^{21}$ & $5$ & $ 1^1 1^{25}$ & $17$ & $ 1^1 1^1 3^8$ & 17 & 2
& $x^{4} + 4 x^{3} + 8 x^{2} + 4 x + 1$\cr
\midrule
$\PGL_2(\FF_{27})$ 
& $3^{39}\cdot 41^{18}$ & $3$ & $ 1^1 1^{27}$ & $41$ & $ 1^1 1^3 (2^3)^4$ & 41 & 2
& $x^{3} + x^{2} - 5 x - 1$\cr
& $3^{39}\cdot 47^{18}$ & $3$ & $ 1^1 1^{27}$ & $47$ & $ 1^1 1^3 (2^3)^4$ & 47 & 2
& $x^{4} - x^{3} - 5 x^{2} + 5 x - 1$\cr
& $3^{39}\cdot 53^{18}$ & $3$ & $ 1^1 1^{27}$ & $53$ & $ 1^1 1^3 (2^3)^4$ & 53 & 2
& $x^{3} + x^{2} - 3 x - 1$\cr
\bottomrule
\end{tabular}
\end{center}
%\newline
%\newline


\begin{thebibliography}{88}

\bibitem{Magma}
W.\ Bosma, J.J.\ Cannon, and C.E.\ Playoust,
{\it The Magma algebra system I: the user language},
J.\ Symbolic Comput.\ {\bf 24} (1997) no.\ 3/4, 235--265.

\bibitem{BoSL2F16}
J.G.\ Bosman,
{\it A polynomial with Galois group $\SL_2(\FF_{16})$},
LMS J.\ Comput.\ Math.\ {\bf 10} (2007) 378--388.

\bibitem{BoComp}
J.G.\ Bosman,
{\it Computations with modular forms and Galois representations},
pp.\ 129--157 of~\cite{TheBigBook}.

\bibitem{BoLvl1}
J.G.\ Bosman,
{\it Polynomials for projective representations of level one forms},
pp.\ 159--172 of~\cite{TheBigBook}.

\bibitem{Peter}
P.J.\ Bruin,
{\it Modular curves, Arakelov theory, algorithmic applications},
Ph.D.\ thesis, Universiteit Leiden, 2010.
{\tt http://hdl.handle.net/1887/15915}

\bibitem{Buzz}
K.\ Buzzard,
{\it On level-lowering for mod 2 representations}.  
Math.\ Res.\ Lett.\  {\bf 7}  (2000)  no.~1, 95--110.

%\bibitem{BuSt}
%K.\ Buzzard and W.A.\ Stein, 
%{\it A mod five approach to modularity of icosahedral Galois representations},
%Pacific J.\ Math.\ {\bf 203} (2002) no.\ 2, 265--282.

\bibitem{Cam}
P.J.\ Cameron, 
{\it Finite permutation groups and finite simple groups},
Bull.\ London Math.\ Soc.\ {\bf 13} (1981) no.\ 1, 1--22.

\bibitem{TheBigBook}
J-M.\ Couveignes, S.J.\ Edixhoven (eds.),
{\it Computational aspects of modular forms and Galois representations},
Ann.\ Math.\ Stud.\ {\bf 176}, Princeton University Press, 2011.

\bibitem{DieuWie}
L.V.\ Dieulefait and G.\ Wiese,
{\it On modular forms and the inverse Galois problem},
Trans.\ Amer.\ Math.\ Soc.\ {\bf 363} (2011) 4569--4584.

\bibitem{EdixWt}
S.J.\ Edixhoven,
{\it The weight in Serre's conjectures on modular forms},
Invent.\ Math.\ {\bf 109} (1992) no.\ 3, 563--594.

\bibitem{GeKl}
K.\ Gei\ss ler and J.\ Kl\"uners,
{\it Galois group computation for rational polynomials},
J.\ Symbolic Comput.\ {\bf 30} (2000) no.\ 6, 653--674.

\bibitem{KhaWi}
C.\ Khare and J-P. Wintenberger, 
{\it Serre's modularity conjecture (I, II)},
Invent.\ Math.\ {\bf 178} (2009) no.\ 3, 485--504, 505--586.

\bibitem{KimVer}
I.\ Kiming and H.A.\ Verrill,
{\it On modular mod $\ell$ representations with exceptional images},
J.\ Number Theory {\bf 110} (2005) 236--266.

\bibitem{Kisin}
M.\ Kisin,
{\it Modularity of $2$-adic Barsotti-Tate representations},
Invent.\ Math.\ {\bf 178} (2009) no.\ 3, 587--634.

\bibitem{KlMa}
J.\ Kl\"uners and G.\ Malle,
{\it Explicit Galois realization of transitive groups of degree up to~$15$},
J.\ Symbolic Comput. {\bf 30} (2000) no.\ 6, 675--716.

\bibitem{KooStWie}
K.T-L.\ Koo, W.A.\ Stein, and G.\ Wiese,
{\it On the generation of the coefficient field of a newform by a single
Hecke eigenvalue},
J.\ Th\'eor.\ Nombres Bordeaux {\bf 20} (2008) no.\ 2, 373--384.

\bibitem{Ma1}
G.\ Malle,
{\it Polynome mit Galoisgruppen $\PGL_2(p)$ und $\PSL_2(p)$ \"uber $\QQ(t)$},
Comm. Algebra {\bf 21} (1993) 511--526.

\bibitem{MazEis}
B.\ Mazur,
{\it Modular curves and the Eisenstein ideal},
Publ.\ Math.\ I.H.E.S.\ {\bf 47} (1977) 33--186.

\bibitem{MoTa}
H.\ Moon and Y.\ Taguchi,
{\it Refinement of Tate's discriminant bound and non-existence theorems for 
mod $p$ Galois representations},
Documenta Math. Extra Volume Kato (2003) 641--654.

\bibitem{Pari}
The PARI-group,
{\it PARI/GP mathematics software (version~2.4.3)}, 
Bordeaux, 2010. {\tt http://pari.math.u-bordeaux.fr/}

\bibitem{RiSt}
K.A.\ Ribet and W.A.\ Stein,
{\it Lectures on Serre's conjectures},
in: Arithmetic algebraic geometry (Park City, UT, 1999),
Amer. Math. Soc., Providence, RI, 2001, 143--232.

\bibitem{SeWt1}
J-P.\ Serre,
{\it Modular forms of weight one and Galois representations},
in: Algebraic number fields: $L$-functions and Galois properties 
(A.\ Fr\"ohlich, ed.), Academic Press, London, 1977, 193--268.

\bibitem{SeConj}
J-P.\ Serre,
{\it Sur les repr\'esentations modulaires de degr\'e 2 de 
$\Gal(\overline{\QQ}/\QQ)$},
Duke Math J.\ {\bf 54} (1987) no.\ 1, 179--230.

\bibitem{Soi}
L.\ Soicher,
{\it The computation of Galois groups},
Master's thesis, Concordia University, Montreal, 1981.

\bibitem{StComp}
W.A.\ Stein,
{\it Modular forms, a computational approach},
Graduate Studies in Mathematics {\bf 79},
Amer.\ Math.\ Soc., Providence, RI, 2007.

\bibitem{Sage}
W.A.\ Stein et~al., 
{\it Sage mathematics software (version~4.7.1)}, The Sage Development Team,
2011. {\tt http://www.sagemath.org/}

\bibitem{Stu}
J.\ Sturm,
{\it On the congruence of modular forms}, 
Lecture Notes in Mathematics {\bf 1240} (1987) 275--280.

\bibitem{Suz}
M.\ Suzuki,
{\it Group Theory I}, Grundlehren der mathematischen Wissenschaften {\bf 247}, 
Springer-Verlag, New York, 1982.

\bibitem{WieGal}
G.\ Wiese,
{\it On projective linear groups over finite fields as Galois groups 
over the rational numbers},
in: Modular forms on Schiermonnikoog 
(S.J.\ Edixhoven, G.B.M.\ van der Geer, B.J.J.\ Moonen, eds.),
Cambridge University Press, 2008, 343--350.
\end{thebibliography}
\end{document}